\documentclass[preprint]{elsarticle}
\usepackage[top=2cm,bottom=3cm,left=3cm,right=2cm]{geometry}
\usepackage[utf8]{inputenc}
\usepackage{amsmath}
\usepackage{amssymb}
\usepackage{amsthm}
\usepackage{wrapfig}
\usepackage{graphicx}
\date{}
\newtheorem{theorem}{Theorem}
\newtheorem{conjecture}{Conjecture}

\begin{document}
\begin{frontmatter}
\title{Interval edge-colorings of $K_{1,m,n}$}

\author[ju]{A. Grzesik}
\ead{Andrzej.Grzesik@uj.edu.pl}
\author[ysu]{H. Khachatrian}
\ead{hrant@egern.net}

\address[ju]{Theoretical Computer Science Department, Faculty of Mathematics and
Computer Science, Jagiellonian University, ul.~Prof.~St.~Łojasiewicza 6, 30-348
Kraków, Poland}
\address[ysu]{Department of Informatics and Applied Mathematics, Yerevan State
University, Yerevan, 0025, Armenia}

\begin{abstract}
An edge-coloring of a graph $G$ with colors $1,\ldots,t$ is an interval
$t$-coloring if all colors are used, and the colors of edges incident to each
vertex of $G$ are distinct and form an interval of integers. A graph $G$ is
interval colorable if it has an interval $t$-coloring for some positive integer
$t$. In this note we prove that $K_{1,m,n}$ is interval colorable if and only if
$\gcd(m+1,n+1)=1$, where $\gcd(m+1,n+1)$ is the greatest common divisor of $m+1$
and $n+1$. It settles in the affirmative a conjecture of Petrosyan.
\end{abstract}

\end{frontmatter}

\section{Introduction}\

All graphs in this paper are finite, undirected and have no loops or multiple
edges. Let $V(G)$ and $E(G)$ denote the sets of vertices and edges of a graph
$G$, respectively. The degree of a vertex $v \in V(G)$ is denoted by $d(v)$, the
maximum degree of $G$ by $\Delta(G)$ and the edge-chromatic number of $G$ by
$\chi'(G)$. The terms, notations and concepts that we do not define can be found
in \cite{West}.

A \emph{proper edge-coloring} of graph $G$ is a coloring of the edges of $G$
such that no two adjacent edges receive the same color. If $\alpha$ is a proper
coloring of $G$ and $v \in V(G)$, then $S(v,\alpha)$ (\emph{spectrum} of a
vertex $v$) denotes the set of colors of edges incident to $v$. A proper
edge-coloring of a graph $G$ with colors $1,\ldots,t$ is an \emph{interval
$t$-coloring} if all colors are used, and for any vertex $v$ of $G$, the set
$S(v, \alpha)$ is an interval of integers. A graph $G$ is \emph{interval
colorable} if it has an interval $t$-coloring for some positive integer $t$. The
set of all interval colorable graphs is denoted by $\mathfrak{N}$. For a graph
$G \in \mathfrak{N}$, the least and the greatest values of $t$ for which $G$ has
an interval $t$-coloring are denoted by $w(G)$ and $W(G)$, respectively.

The concept of interval edge-coloring was introduced by Asratian and Kamalian
\cite{AsratianKamalian1987}. In
\cite{AsratianKamalian1987,AsratianKamalian1994}, they proved that if $G$ is
interval colorable, then $\chi'(G)=\Delta(G)$. Moreover, they also showed that
if $G$ is a triangle-free graph and $G\in\mathfrak{N}$, then $W(G) \leq
|V(G)|-1$. In \cite{Kamalian1989}, Kamalian investigated interval edge-colorings
of complete bipartite graphs and trees. Later, Kamalian \cite{Kamalian1990}
showed that if $G$ is a connected graph and $G \in\mathfrak{N}$, then $W(G) \leq
2|V(G)|-3$. This upper bound was improved by Giaro, Kubale and Malafiejski in
\cite{GiaroKubMal2001}, where they proved that if $G$ ($|V(G)| \geq 3$) is a
connected graph and $G \in\mathfrak{N}$, then $W(G) \leq 2|V(G)|-4$. Recently,
Kamalian and Petrosyan \cite{KamalianPetrosyan2012} showed that if $G$ is a
connected $r$-regular graph ($|V(G)| \geq 2r+2$) and $G \in\mathfrak{N}$, then
$W(G) \leq 2|V(G)|-5$. Interval edge-colorings of planar graphs were considered
by Axenovich in \cite{Axenovich2002}, where she proved that if $G$ is a
connected planar graph and $G\in\mathfrak{N}$, then $W(G) \leq
\frac{11}{6}|V(G)|$. In
\cite{Petrosyan2010}, Petrosyan investigated interval colorings of complete
graphs and $n$-dimensional cubes. In particular, he proved that if
$n\leq t\leq \frac{n\left(n+1\right)}{2}$, then the $n$-dimensional
cube $Q_{n}$ has an interval $t$-coloring. Recently, Petrosyan,
the second author and Tananyan \cite{PetrosyanKhachatrianTananyan2013} showed that the
$n$-dimensional
cube $Q_{n}$ has an interval $t$-coloring if and only if $n\leq
t\leq \frac{n\left(n+1\right)}{2}$. In \cite{Sevastjanov1990}, Sevast'janov
proved that it is an $NP$-complete problem to decide whether a
bipartite graph has an interval coloring or not.

Interval edge-colorings of some special cases of complete multipartite graphs
were first considered by Kamalian in \cite{Kamalian1989}, where he proved the
following
\begin{theorem}
For any $m,n\in \mathbb{N}$, $K_{m,n} \in \mathfrak{N}$ and
\begin{description}
\item[(i)] $w(K_{m,n}) = m+n-\gcd(m,n)$
\item[(ii)] $W(K_{m,n}) = m+n-1$
\item[(iii)] if $w(K_{m,n}) \leq t \leq W(K_{m,n})$, then $K_{m,n}$ has an
interval $t$-coloring.
\end{description}
\end{theorem}

Also, he showed that complete graphs are interval colorable if and only if the
number of vertices is even. Moreover, for any $n\in\mathbb{N}$, $w(K_{2n}) =
2n-1$. For a lower bound on $W(K_{2n})$, Kamalian obtained the following result:
\begin{theorem}
For any $n\in \mathbb{N}$, $W(K_{2n}) \geq 2n-1 + \lfloor \log_2{(2n-1)}
\rfloor$.
\end{theorem}
Later, Petrosyan \cite{Petrosyan2010} improved this lower bound for $W(K_{2n})$:
\begin{theorem}
If $n=p2^{q}$, where $p$ is odd and $q$ is nonnegative, then
\begin{center}
$W\left(K_{2n}\right)\geq 4n-2-p-q$.
\end{center}
\end{theorem}
In the same paper he also conjectured that this lower bound is the exact value
of $W(K_{2n})$. He verified this conjecture for $n \leq 4$, but the conjecture
was disproved by the second author in \cite{Khachatrian2012}.

Another special case of complete multipartite graphs was considered by Feng and
Huang in \cite{FengHuang2007}, where they proved the following
\begin{theorem}
\label{FengHuang}
For any $n\in\mathbb{N}$, $K_{1,1,n}\in \mathfrak{N}$ if and only if $n$ is
even.
\end{theorem}

Recently, Petrosyan investigated interval edge-colorings of complete
multipartite graphs. In particular, he proved \cite{Petrosyan2012} the following
result:
\begin{theorem}
If $K_{n,\ldots,n}$ is a complete balanced $k$-partite graph, then
$K_{n,\ldots,n} \in \mathfrak{N}$ if and only if $nk$ is even. Moreover, if $nk$
is even, then $w(K_{n,\ldots,n}) = n(k-1)$ and $W(K_{n,\ldots,n}) \geq \left(
\frac{3}{2}k-1 \right)n-1$.
\end{theorem}

In "Cycles and Colorings 2012" workshop Petrosyan presented several conjectures
on interval edge-colorings of complete multipartite graphs. In particular, he
posed the following
\begin{conjecture}
For any $m,n\in\mathbb{N}$, $K_{1,m,n}\in\mathfrak{N}$ if and only if
$\gcd(m+1,n+1)=1$.
\end{conjecture}

In this note we prove this conjecture, which also generalizes Theorem
\ref{FengHuang}.

\section{Main result}\

We denote the bipartition of $K_{m,n}$ by $(U, V)$, where $U = \left\{ u_0, u_1,
\ldots, u_{m-1} \right\}$ and $V = \left\{ v_0, v_1, \ldots, v_{n-1} \right\}$.
The interval edge-coloring $\alpha_{m,n}$ of $K_{m,n}$ with maximum number of
colors is given the following way:
\begin{center}
$ \alpha_{m,n}(u_iv_j) = i+j+1$, where
$0 \leq i \leq m-1$, \hspace{3pt} $0 \leq j \leq n-1$.
\end{center}
$K_{1,m,n}$ is a complete tripartite graph that can be viewed as a $K_{m,n}$
plus one additional vertex connected to all other vertices. In this paper we
prove that if $m+1$ and $n+1$ are coprime, then it is possible to extend the
$\alpha_{m,n}$ coloring of $K_{m,n}$ to an interval edge-coloring of
$K_{1,m,n}$. Then we prove that if $\gcd(m+1,n+1) > 1$, then $K_{1,m,n}$ is not
interval colorable.

\begin{figure}[h]
\centering
\includegraphics[scale=0.9]{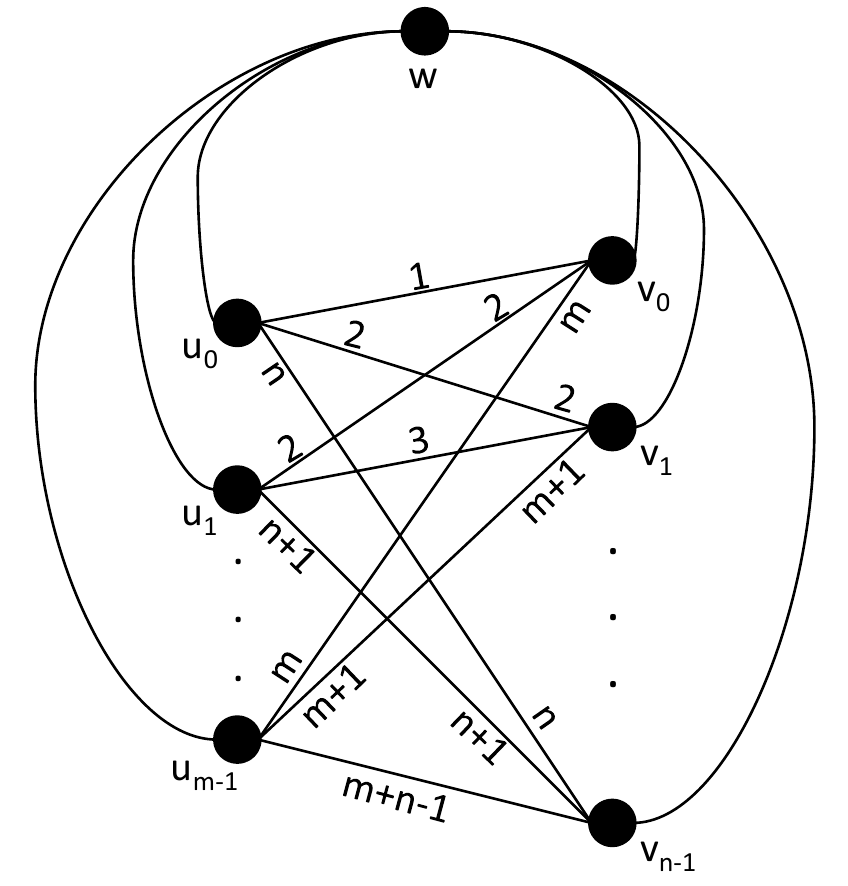}
\caption{$K_{1,m,n}$ with $\alpha_{m,n}$ coloring.}
\label{graphK1mn}
\end{figure}

Spectrums of the vertices for $\alpha_{m,n}$ coloring are the following:
\begin{center}
\begin{tabular}{cc}
$ S(u_i, \alpha_{m,n}) = \{i+1, \ldots, i+n \}$, & $0 \leq i \leq m-1$ \\
$ S(v_j, \alpha_{m,n}) = \{j+1, \ldots, j+m \}$, & $0 \leq j \leq n-1$
\end{tabular}
\end{center}

We construct $K_{1,m,n}$ by adding a new vertex $w$ to $K_{m,n}$ and joining it
with all the remaining vertices.
\begin{align*}
V(K_{1,m,n}) & = V(K_{m,n}) \cup \{w\} \\
E(K_{1,m,n}) & = E(K_{m,n}) \cup \{u_iw\ |\ 0 \leq i \leq m-1 \} \cup \{v_jw\ |\
0 \leq j \leq n-1 \}
\end{align*}

\begin{theorem}
If $\gcd(m+1,n+1)=1$, then $K_{1,m,n}$ has an interval $(m+n)$-coloring.
\end{theorem}
\begin{proof}
We color the edges $u_iv_j$ of $K_{1,m,n}$ the same way as in $\alpha_{m,n}$. In
order to prove the theorem it is sufficient to show that it is possible to color
the remaining edges in a way that the following conditions are met:
\begin{description}
\item[(1)] spectrums of vertices $u_i$ and $v_j$ remain intervals of integers
\item[(2)] spectrum of the vertex $w$ is also an interval of integers
\end{description}
We construct an auxillary bipartite graph $H$ which has a bipartition $(B,C)$
where $B$ corresponds to the edges $u_iw$ and $v_jw$, and $C$ corresponds to the
colors that will be used to color those edges.
\begin{center}
$B = \left\{u'_i\ |\ 0 \leq i \leq m-1\right\} \cup \left\{v'_i\ |\ 0 \leq i
\leq n-1\right\}$
\end{center}
where $u'_i$ and $v'_j$ correspond, respectively, to $u_iw$ and $v_jw$ in
$E(K_{1,m,n})$.

\begin{center}
$C = \left\{c_k\ |\ 1 \leq k \leq m+n\right\}$
\end{center}
where $c_k$ corresponds to the color $k$. We will join the vertices $b\in B$ and
$c_k \in C$ if and only if we allow the edge corresponding to $b$ to receive the
color $k$. Note that $|B|=|C|=m+n$.

\begin{figure}[h]
\centering
\includegraphics[width=\textwidth]{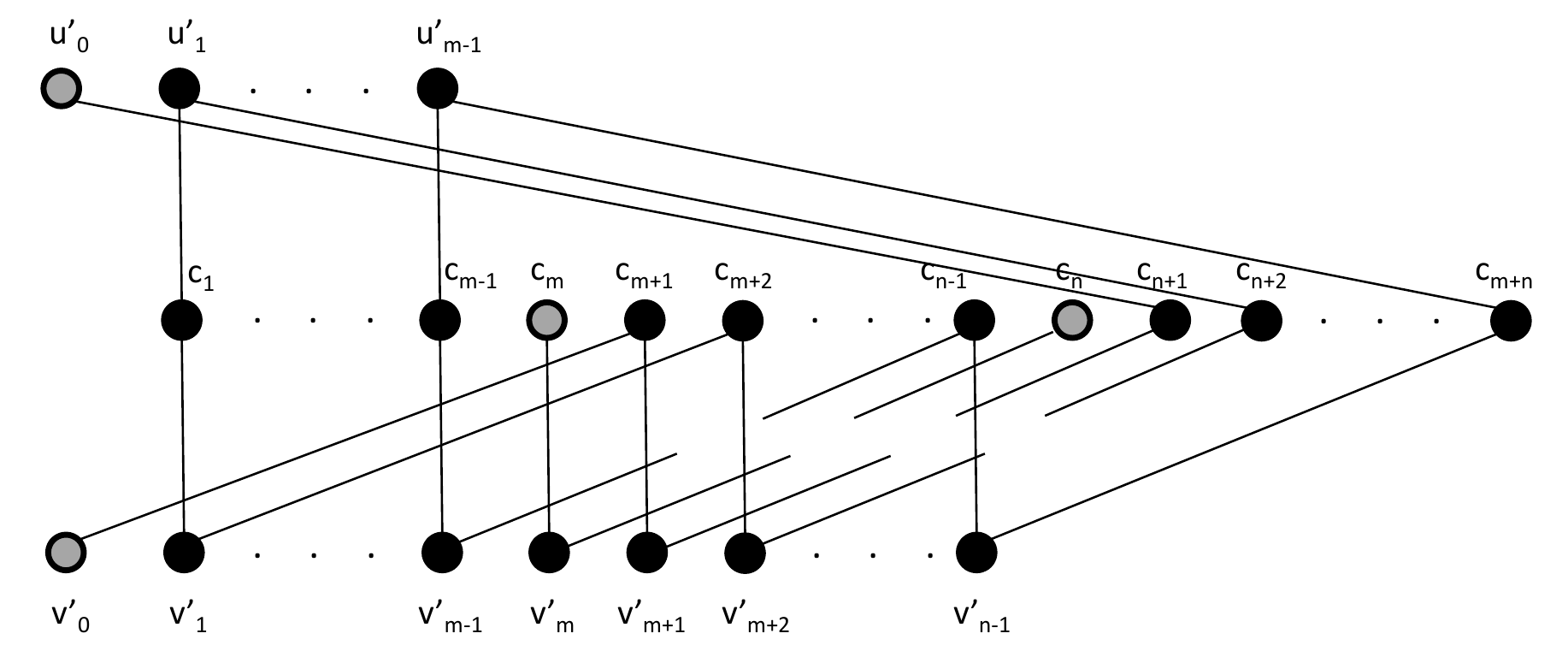}
\caption{Auxillary graph $H$}
\label{graphH}
\end{figure}

$S(u_0,\alpha_{m,n}) = \{1,\ldots,n\}$, so in order to satisfy the condition
\textbf{(1)} the edge $u_0w$ can only receive the color $n+1$ (we don't want to
allow color $0$). Similarly, $v_0w$ can only be colored by $m+1$. For $u_iw$
($1\leq i\leq m-1$) we have two options: either $i$ or $i+n+1$. For $v_jw$ ($1
\leq j \leq n-1$) we allow colors $j$ and $j+m+1$. Therefore,
\begin{align*}
E(H) = &\left\{u'_ic_i\ |\ 1 \leq i \leq m-1 \right\} \cup
\left\{u'_ic_{i+n+1}\ |\ 0 \leq i \leq m-1 \right\} \cup \\
\cup &\left\{v'_jc_j\ |\ 1 \leq j \leq n-1 \right\} \cup
\left\{v'_jc_{j+m+1}\ |\ 0 \leq j \leq n-1 \right\}
\end{align*}

Suppose $M$ is a matching in $H$. For each $bc_k \in M$ we color the edge of
$K_{1,m,n}$ corresponding to vertex $b$ by color $k$. If $M$ is a perfect
matching all remaining edges of $G$ will be colored and all the colors will be
used. So, the spectrum of vertex $w$ will be $\{1,\ldots,m+n\}$ and the
condition \textbf{(2)} will be satisfied. Condition \textbf{(1)} will be
satisfied because of the construction of $H$.

So, to complete the proof we show that $H$ has a perfect matching.

Without loss of generality we can assume that $m<n$. $m=n$ case is excluded
because $\gcd(n+1,n+1) \not= 1$. From the construction of graph $H$ it follows,
that all vertices have a degree $2$, except for four vertices, which have degree
$1$, namely $u'_0$, $v'_0$, $c_m$ and $c_n$. Therefore, $H$ consists of several
even cycles (as it is bipartite) and 2 simple paths. $H$ will have a perfect
matching if both paths have odd length. Therefore, $u'_0$ and $v'_0$ must belong
do distinct paths.

Suppose, to the contrary, that $u'_0$ and $v'_0$ belong to the same path $P$. We
introduce a coordinate system and embed the graph $H$ the following way:
coordinates for vertices $u'_i$, $v'_i$ and $c_i$ are $(i, 1)$, $(i, -1)$ and
$(i, 0)$ respectively (Figure \ref{graphH}). We split the edges of $P$ into two
groups. First group contains edges of type $u'_ic_{i+n+1}$ and $v'_ic_i$ and
second group contains the remaining edges. Note that each edge of the path $P$
has only neighbors from the other group. If we begin to traverse the path $P$
starting at the vertex $u'_0$ we go down only along non-vertical edges and go up
only along vertical edges. Suppose we moved along the edges of type
$u'_ic_{i+n+1}$ $a$ times, each time increasing the abscissa by $n+1$, and moved
along the edges of type $c_{j+m+1}v'_j$ $b$ times, each time decreasing the
abscissa by $m+1$. Moving along the vertical edges does not change the abscissa.
The last vertex of the path is $v'_0$ which has an abscissa of $0$, therefore we
obtain the following equation:
\begin{center}
$a(n+1) - b(m+1) = 0$
\end{center}
Note that $a\leq m$ and $b \leq n$. Moreover, $\gcd(m+1, n+1)=1$, so we have
$a=b=0$. The path $P$ has no edges, which is a contradiction.
\end{proof}

\begin{theorem}
If $\gcd(m+1,n+1)>1$ then $K_{1,m,n}$ is not interval colorable.
\end{theorem}
\begin{proof}
Suppose, to the contrary, that $\gcd(m+1,n+1) = d >1$ and $\beta$ is an
interval-edge coloring of $K_{1,m,n}$. We call an edge $e \in E(K_{1,m,n})$ a
"$d$-edge" if $\beta(e) = dx$ for some $x \in \mathbb{Z}$. We denote by $D(v)$
the number of $d$-edges incident to the vertex $v \in V(K_{1,m,n})$.

Without loss of generality we assume that $S(w,\beta) = \{1,\ldots,m+n\}$
(otherwise we will shift the colors of all edges by the same amount in a way
that the spectrum of vertex $w$ starts with color $1$). Therefore,
\begin{align*}
D(w) = \left\lfloor \frac{m+n}{d} \right\rfloor = \frac{m+n+2}{d}-1
\end{align*}
$\left| S(u_i, \beta) \right| = n+1$ for all $0\leq i \leq m-1$, and $\left|
S(v_j, \beta) \right| = m+1$ for all $0\leq j \leq n-1$. Therefore,
\begin{align*}
D(u_i) = \frac{n+1}{d},\ 0\leq i \leq m-1\\
D(v_j) = \frac{m+1}{d},\ 0\leq j \leq n-1
\end{align*}
The sum of $D(v)$ over all vertices $v \in V(K_{1,m,n})$ must give twice the
number of $d$-edges in the graph.
\begin{align*}
D = \sum_{v \in V(K_{1,m,n})}{D(v)} = \frac{m+n+2}{d}-1 + m\frac{n+1}{d} +
n\frac{m+1}{d} = \frac{2(m+1)(n+1)}{d} - 1
\end{align*}
This is a contradiction, because $D$ is an odd number.
\end{proof}

\section{Future work}\

There are two natural ``next steps'' after finding the condition for
colorability of $K_{1,m,n}$. First one is to find the exact number of colors needed
to color $K_{1,m,n}$ and the second is to generalize the statement to
colorability of $K_{k,m,n}$ for $k>1$. Past work on those problems led us to the following conjectures.

\begin{conjecture}
Graph $K_{1,m,n}$ has an interval $t$-coloring if and only if $t=m+n$ and $\gcd(m+1,n+1)=1$.
\end{conjecture}

\begin{conjecture}
Graph $K_{k,m,n}$, where $k \le m \le n$ and $n > k+m$ is interval colorable if and only if graph $K_{k,m,n-k-m}$ is interval colorable.
\end{conjecture}

\begin{conjecture}
Graph $K_{k,m,n}$, where $k \le m \le n$ and $n \le k+m$ is interval colorable if and only if the sum $k+m+n$ is even.
\end{conjecture}

It can be seen that the last two conjectures generalize the proved statement for colorability of $K_{1,m,n}$.

\end{document}